\documentclass[11pt,english]{article}
\usepackage[T1]{fontenc}
\usepackage[latin9]{inputenc}
\usepackage{amsthm}
\usepackage{amsmath}
\usepackage{amssymb}

\makeatletter

\newcommand{\noun}[1]{\textsc{#1}}
\DeclareRobustCommand{\greektext}{%
  \fontencoding{LGR}\selectfont\def\encodingdefault{LGR}}
\DeclareRobustCommand{\textgreek}[1]{\leavevmode{\greektext #1}}
\DeclareFontEncoding{LGR}{}{}

\theoremstyle{plain}
\theoremstyle{plain}
\newtheorem{thm}{Theorem}
  \theoremstyle{plain}
  \newtheorem{cor}[thm]{Corollary}
  \theoremstyle{plain}
  \newtheorem*{conjecture*}{Conjecture}
  \theoremstyle{plain}
  \newtheorem{prop}[thm]{Proposition}
  \theoremstyle{plain}
  \newtheorem*{thm*}{Theorem}
  \theoremstyle{plain}
  \newtheorem{lem}[thm]{Lemma}

\usepackage{mathrsfs}\usepackage{amsthm}\usepackage{amscd}

\usepackage{hyperref}
\usepackage{graphics}
\usepackage{a4wide}
\@ifundefined{definecolor}
 {\usepackage{color}}{}

\newcounter{EQNR}
\renewcommand{\contentsname}{Contents\\
{\footnotesize\normalfont(A table of contents should normally not
be included)}}

\makeatother

\usepackage{babel}

\begin{document}

\title{Spectral zeta functions of graphs and the Riemann zeta function in
the critical strip}

\author{Fabien Friedli and Anders Karlsson%
\thanks{The authors were supported in part by the Swiss NSF grant 200021 132528/1.%
}}
\maketitle
\begin{abstract}
We initiate the study of spectral zeta functions $\zeta_{X}$ for
finite and infinite graphs $X$, instead of the Ihara zeta function,
with a perspective towards zeta functions from number theory and connections
to hypergeometric functions. The Riemann hypothesis is shown to be
equivalent to an approximate functional equation of graph zeta functions.
The latter holds at all points where Riemann's zeta function $\zeta(s)$
is non-zero. This connection arises via a detailed study of the asymptotics
of the spectral zeta functions of finite torus graphs in the critcal
strip and estimates on the real part of the logarithmic derivative
of $\zeta(s)$. We relate $\zeta_{\mathbb{Z}}$ to Euler's beta integral
and show how to complete it giving the functional equation $\xi_{\mathbb{Z}}(1-s)=\xi_{\mathbb{Z}}(s)$.
This function appears in the theory of Eisenstein series although
presumably with this spectral intepretation unrecognized. In higher
dimensions $d$ we provide a meromorphic continuation of $\zeta_{\mathbb{Z}^{d}}(s)$
to the whole plane and identify the poles. From our aymptotics several
known special values of $\zeta(s)$ are derived as well as its non-vanishing
on the line $Re(s)=1$. We determine the spectral zeta functions of
regular trees and show it to be equal to a specialization of Appell's
hypergeometric function $F_{1}$ via an Euler-type integral formula
due to Picard. 
\end{abstract}

\section{Introduction}

In order to study the Laplace eigenvalues $\lambda_{n}$ of bounded
domains $D$ in the plane, Carleman employed the function \[
\zeta_{D}(s)=\sum_{n=1}^{\infty}\frac{1}{\lambda_{n}^{s}}\]
taking advantage of techniques from the theory of Dirichlet series
including Ikehara's Tauberian theorem \cite{Ca34}. This was followed-up
in \cite{P39}, and developed further in \cite{MP49} for the case
of compact Riemannian manifolds. These zeta functions have since played
a role in the definitions of determinants of Laplacians and analytic
torsion, and they are important in theoretical physics \cite{Ha77,El12,RV15}.
For graphs it has been popular and fruitful to study the Ihara zeta
function, which is an analog of the Selberg zeta function in turn
modeled on the Euler product of Riemann's zeta function. Serre noted
that Ihara's definition made sense for any finite graph and this suggestion
was taken up and developed by Sunada, Hashimoto, Hori, Bass and others,
see \cite{Su86,Te10}.

The present paper has a three-fold objective. First, we advance the
study of spectral zeta functions of graphs, instead of the Ihara zeta
function. We do this even for infinite graphs where the spectrum might
be continuous. For the most fundamental infinite graphs, this study
leads into the theory of hypergeometric function in several variables,
such as those of Appell, and gives rise to several questions. 

Second, we study the asymptotics of spectral zeta functions for finite
torus graphs as they grow to infinity, in a way similar to what is
often considered in statistical physics (see for example \cite{DD88})\emph{.
}The study of limiting sequences of graphs is also a subject of significant
current mathematical interest, see \cite{Lo12,Ly10,LPS14}.\emph{
}Terms appearing in our asymptotic expansions are zeta functions of
lattice graphs and of continuous torus which are Epstein zeta function
from number theory. This relies to an important extent on the work
of Chinta, Jorgenson, and the second-named author \cite{CJK10}, in
particular we quote and use without proof several results established
in this reference.

Third, we provide a new perspective on some parts of analytic number
theory, in two ways. In one way, this comes via replacing partial
sums of Dirichlet series by zeta functions of finite graphs. Although
the latter looks somewhat more complicated, they have more structure,
being a spectral zeta function, and are decidedly easier in some respects.
We show the equivalence of the Riemann hypothesis with a conjectural
functional equation for graph spectral zeta functions, and this seems
substantially different from other known reformulations of this important
problem \cite{RH07}. In a second way, the spectral zeta function
of the graph $\mathbb{Z}$ enjoys properties analogous to the Riemann
zeta function, notably the relation $\xi_{\mathbb{Z}}(1-s)=\xi_{\mathbb{Z}}(s)$,
and it appears \emph{incognito }as fudge factor in a few instances
in the classical theory, such as in the Fourier development of Eisenstein
series.

For us, a spectral zeta function $\zeta_{X}$ of a space $X$ is the
Mellin transform of the heat kernel of $X$ at the origin, removing
the trivial eigenvalue if applicable, and divided by a gamma factor
(\emph{cf.} \cite{JL12}). Alternatively one can define this function
by an integration against the spectral measure.

Consider a sequence of discrete tori $\mathbb{Z}^{d}/A_{n}\mathbb{Z}^{d}$
indexed by $n$ and where the matrices $A_{n}$ are diagonal with
entries $a_{i}n$, and integers $a_{i}>0.$ The matrix $A$ is the
diagonal matrix with entries $a_{i}$. We show the following for any
dimension $d\geq1$:
\begin{thm}
\label{thm:Zd}The following asymptotic expansion as $n\rightarrow\infty$
is valid for $Re(s)<d/2+1,$ and $s\neq d/2$,\[
\zeta_{\mathbb{Z}^{d}/A_{n}\mathbb{Z}^{d}}(s)=\zeta_{\mathbb{Z}^{d}}(s)\det A\, n^{d}+\zeta_{\mathbb{R}^{d}/A\mathbb{Z}^{d}}(s)n^{2s}+o(n^{2s}).\]

\end{thm}
The formula reflects that as $n$ goes to infinity the finite torus
graph can be viewed as converging to $\mathbb{Z}^{d}$ on the one
hand, and rescaled to the continuous torus $\mathbb{R\mathrm{^{d}/\mathbb{Z}^{d}}}$
on the other hand. For $Re(s)>d/2$ one has\begin{equation}
\lim_{n\rightarrow\infty}\frac{1}{n^{2s}}\zeta_{\mathbb{Z}^{d}/A_{n}\mathbb{Z}^{d}}(s)=\zeta_{\mathbb{R}^{d}/A\mathbb{Z}^{d}}(s),\label{eq:zetalimit}\end{equation}

as already shown in \cite{CJK10}, see also section \ref{sec:Asymptotics}
below. One can verify that it is legitimate to differentiate in the
asymptotics in Theorem \ref{thm:Zd} and if we then set $s=0$, we
recover as expected the main asymptotic formula in \cite{CJK10} in
the case considered. The asymptotics of the determinant of graph Laplacians
is a topic of significant interest, see \cite[Conclusion]{RV15} for
a recent discussion from the point of view of quantum field theory,
and see \cite{Lü02} for related determinants in the context of $L^{2}$-invariants.

We now specialize to the case $d=1$. In particular, the spectral
zeta function of the finite cyclic graph $\mathbb{Z\mathrm{/}}n\mathbb{Z}$
(see e.g. \cite{CJK10} for details and section \ref{sec:Spectral-zeta-functions})
is \[
\zeta_{\mathbb{Z\mathrm{/}}n\mathbb{Z}}(s)=\frac{1}{4^{s}}\sum_{k=1}^{n-1}\frac{1}{\sin^{2s}(\pi k/n)}.\]

The spectral zeta function of the graph $\mathbb{Z}$ is\[
\zeta_{\mathbb{Z}}(s)=\frac{1}{\Gamma(s)}\int_{0}^{\infty}e^{-2t}I_{0}(2t)t^{s}\frac{dt}{t},\]
where it converges, which it does for $0<Re(s)<1/2$. From this definition
it is not immediate that its meromorphic continuation admits a functional
equation much analogous to classical zeta functions:
\begin{thm}
Let the completed zeta function for $\mathbb{Z}$ be defined as \[
\xi_{\mathbb{Z}}(s)=2^{s}\cos(\pi s/2)\zeta_{\mathbb{Z}}(s/2).\]
Then this is an entire function that satisfies for all $s\in\mathbb{C}$
\[
\xi_{\mathbb{Z}}(s)=\xi_{\mathbb{Z}}(1-s).\]

\end{thm}
This raises the question: Are there other spectral zeta functions
of graphs with similar properties? 

The function $\zeta_{\mathbb{Z}}$ actually appears implicitly in
classical analytic number theory. Let us exemplify this point. To
begin with\[
\zeta_{\mathbb{Z}}(s)=\frac{1}{4^{s}\sqrt{\pi}}\frac{\Gamma(1/2-s)}{\Gamma(1-s)},\]
which is the crucial fact behind the result above. Now, in the main
formula of Chowla-Selberg in \cite{SC67} the following term appears:\[
\frac{2^{2s}a^{s-1}\sqrt{\pi}}{\Gamma(s)\Delta^{s-1/2}}\zeta(2s-1)\Gamma(s-1/2).\]
Here lurks $\zeta_{\mathbb{Z}}(1-s)$, not only by correctly combining
the two gamma factors, but also incorporating the factor $2^{2s}$
and explaining the appearance of $\sqrt{\pi}$. In other words, the
term above equals\[
\frac{4\pi a^{s-1}}{\Delta^{s-1/2}}\zeta(2s-1)\zeta_{\mathbb{Z}}(1-s).\]

Upon dividing by the Riemann zeta function $\zeta(s)$, this term
is called scattering matrix (function) in the topic of Fourier expansions
of Eisenstein series and is complicated or unknown for discrete groups
more general than $SL(2,\mathbb{Z})$, see \cite[section 15.4]{IK04}
and \cite{Mü08}. We believe that the interpretation of such fudge
factors as spectral zeta functions is new and may provide some insight
into how such factors arise more generally.

The Riemann zeta function is essentially the same as the spectral
zeta function of the circle $\mathbb{R\mathrm{/\mathbb{Z}}}$, more
precisely one has\begin{equation}
\zeta_{\mathbb{R\mathrm{/\mathbb{Z}}}}(s)=2(2\pi)^{-2s}\zeta(2s).\label{eq:Riemannzeta}\end{equation}

Here is a specialization of Theorem \ref{thm:Zd} to $d=1$ with explicit
functions and some more precision:
\begin{thm}
\label{thm:d=00003D1}For $s\neq1$ with $Re(s)<3$ it holds that\[
\sum_{k=1}^{n-1}\frac{1}{\sin^{s}(\pi k/n)}=\frac{1}{\sqrt{\pi}}\frac{\Gamma(1/2-s/2)}{\Gamma(1-s/2)}n+2\pi^{-s}\zeta(s)n^{s}+o(n^{s})\]
as $n\rightarrow\infty$. In the critical strip, $0<Re(s)<1,$ more
precise asymptotics can be found, such as\[
\sum_{k=1}^{n-1}\frac{1}{\sin^{s}(\pi k/n)}=\frac{1}{\sqrt{\pi}}\frac{\Gamma(1/2-s/2)}{\Gamma(1-s/2)}n+2\pi^{-s}\zeta(s)n^{s}+\frac{s}{3}\pi^{2-s}\zeta(s-2)n^{s-2}+o(n^{s-2})\]
as $n\rightarrow\infty$.
\end{thm}
For example, with $s=0$ the sum on the left equals $n-1$, and the
asymptotic formula hence confirms the well-known values $\Gamma(1/2)=\sqrt{\pi}$
and $\zeta(0)=-1/2$. On the line $Re(s)=1$, the asymptotics is critical
in the sense that the two first terms on the right balance each other
in size as a power of $n$. As a consequence, for all $t\neq0$ we
have that $\zeta(1+it)\neq0$ if and only if \[
\frac{1}{n}\sum_{k=1}^{n-1}\frac{1}{\sin^{1+it}(\pi k/n)}\]
diverges as $n\rightarrow\infty$. The latter sum does indeed diverge.
We do not have a direct proof of this at the moment, but it does follow
from a theorem of Wintner \cite{W47} since the improper integral
$\int\sin^{-1-it}(x)dx$ diverges at $x=0$. So we have that the Riemann
zeta function has no zeros on the line $Re(s)=1$, which is the crucial
input in the standard proof of the prime number theorem. It should
however be said that Wintner's theorem is known to already be intimately
related to the prime number theorem via works of Hardy-Littlewood.

As suggested to us by Jay Jorgenson, one may differentiate the formula
in Theorem \ref{thm:Zd} for $d=1$, as can be verified via the formulas
in section \ref{sec:Asymptotics}, and get a criterion for multiple
zeros: 
\begin{cor}
Let\[
c(s)=\frac{1}{2\sqrt{\pi}}\frac{\Gamma(1/2-s/2)}{\Gamma(1-s/2)}\left(\frac{\Gamma'(1/2-s/2)}{\Gamma(1/2-s/2)}-\frac{\Gamma'(1-s/2)}{\Gamma(1-s/2)}\right)\]
and\[
S(s,n)=c(s)n-\sum_{k=1}^{n-1}\frac{\log(\sin(\pi k/n))}{\sin^{s}(\pi k/n)}.\]
Then $\zeta$ has a multiple zero at $s,$ $0<Re(s)<1$ if and only
if $S(s,n)\rightarrow0$ as $n\rightarrow\infty$, and otherwise $S(s,n)\rightarrow\infty$
as $n\rightarrow\infty$. 
\end{cor}
It is believed that all Riemann zeta zeros are simple. 

Similarily to the above discussion about the prime number theorem,
the Riemann hypothesis has a formulation in terms of the behaviour
of the sum of sines (here we can refer to \cite{So98} for comparison).
It turns out that with some further investigation there is, what we
think, a more intriguing formulation of the Riemann hypothesis. This
is in terms of functional equations and provides perhaps some further
heuristic evidence for its validity. Let\[
h_{n}(s)=(4\pi)^{s/2}\Gamma(s/2)n^{-s}\left(\zeta_{\mathbb{Z}/n\mathbb{Z}}(s/2)-n\zeta_{\mathbb{Z}}(s/2)\right).\]

\begin{conjecture*}
Let $s\in\mathbb{C}$ with $0<Re(s)<1.$ Then\[
\lim_{n\rightarrow\infty}\left|\frac{h_{n}(1-s)}{h_{n}(s)}\right|=1.\]

\end{conjecture*}
This is a kind of asymptotic or approximative functional equation,
and it is true almost everywhere as follows from the asymptotics above:
\begin{cor}
\label{cor:The-conjecture-holds}The conjecture holds in the crticial
strip wherever $\zeta(s)\neq0$.
\end{cor}
So the question is whether it also holds at the Riemann zeros. Note
that, as discussed $\zeta_{\mathbb{Z}}(s/2)$ has a functional equation
of the desired type, $s\longleftrightarrow1-s$, and also $\zeta_{\mathbb{Z}/n\mathbb{Z}}(s/2)$
in an asymptotic sense\emph{, }see\emph{ }section \ref{sec:Approximative-functional-equations}.
Here is the relation to the Riemann hypothesis:
\begin{thm}
\label{thm:conjectureRH}The conjecture is equivalent to the Riemann
hypothesis.
\end{thm}
Section \ref{sec:The-Riemann-hypothesis} is devoted to the proof
of this statement. This relies in particular on properties of the
logarithmic derivative of $\zeta$, in the proof of Lemma \ref{lem:Riemann},
and the Riemann functional equation.

A referee pointed out that it is important to emphasize that the last
few results in one dimension hold with the same proofs for more general
sums, instead of the inverse sine sums coming from cyclic graphs.
More precisely, let $f$ be an analytic function being real and positive
on the open interval $(0,1)$, satisfying $f(z)=f(1-z)$ for any $z\in\mathbb{C}$,
with $f(0)=0,$ $f'(0)>0,$ $f''(0)=0$ and $f^{(3)}(0)\neq0$.

Now let for $0<Re(s)<1$ \[
h_{n}[f](s)=f'(0)^{s}\pi^{-s/2}\Gamma(s/2)n^{-s}\left[\sum_{j=1}^{n-1}\frac{1}{f(j/n)^{s}}-n\int_{0}^{1}\frac{dx}{f(x)^{s}}\right].\]
As in section \ref{sec:dimension one} applying \cite{Si04} one gets\[
h_{n}[f](s)=2\xi(s)-\frac{f^{(3)}(0)}{f'(0)\pi^{2}}\alpha(s)n^{-2}+o(n^{-2})\]
as $n\rightarrow\infty$ and where $\alpha$ is the function appearing
in section \ref{sec:The-Riemann-hypothesis}. So one may formulate
the same conjecture above, and Corollary \ref{cor:The-conjecture-holds}
and Theorem \ref{thm:conjectureRH} hold for $h_{n}[f]$.

\subsection*{Some concluding remarks. }

Why do we think that the study of sums like \[
\sum_{k=1}^{n-1}\frac{1}{\sin^{s}(\pi k/n)}\]
could in some ways be better than the standard Dirichlet series $\sum_{1}^{n}k^{-s}$,
or some other sum of similar type for that matter? For example, it
has been pointed out to us that we could also derive version of Theorems
\ref{thm:d=00003D1} and \ref{thm:conjectureRH} for more general
functions, as described above, for example replacing sine with $x-2x^{3}+x^{4}=x(1-x)(1+x-x^{2})$$.$
In this case the function corresponding to our $\zeta_{\mathbb{Z}}(s)$,
say in the definition of $h_{n}$, would be\[
\int_{0}^{1}\frac{1}{x^{2s}(1-x)^{2s}(1+x-x^{2})^{2s}}dx,\]
which is a less standard function.

Let us now address this legitimate question with several answers that
reinforce each other:
\begin{enumerate}
\item The graph zeta functions are defined in a parallel way to the definition
of Riemann's zeta. Functions arising in this way may have greater
chance to have more symmetries and structure, for example, keep in
mind the remarkable relation\[
\xi_{\mathbb{Z}}(1-s)=\xi_{\mathbb{Z}}(s),\]
which is far from being just an abstract generality. On the other
hand, it is highly unclear whether the integral above satisfies an
analog of this. The funtional equation for $\zeta_{\mathbb{Z}}$ leads
to, see section \ref{sec:Approximative-functional-equations}, an
asymptotic functional relation of the desired type for the completed
finite $1/\sin$ sums: \[
\lim_{n\rightarrow\infty}\frac{1}{n}\left(\xi_{\mathbb{Z}/n\mathbb{Z}}(1-s)-\xi_{\mathbb{Z}/n\mathbb{Z}}(s)\right)=0\]
 in the critical strip. We do not see a similar relation for, say\[
\sum_{k=1}^{n-1}\frac{1}{\left((k/n)(1-k/n)(1+k/n-k^{2}/n^{2}\right)^{2s}}.\]
Since relations when $s\longleftrightarrow1-s$ is at the heart of
the matter for our reformulation of the Riemann Hypothesis, this may
be a certain advantage respectively disadvantage for the choices of
finite sums to consider.
\item The function $\zeta_{\mathbb{Z}}$ admits analytic continuation, functional
equation and a few nice special values. Furthermore, it appears in
the theory of Eisenstein series as observed above in a way that is
difficult to deny, and in our opinon, unwise to dismiss. 
\item Symmetric functions of graph eigenvalues often have combinatorial
interpretations as counting something (starting with Kirchhoff's matrix
tree theorem), see our section \ref{sub:special-combinatorics} for
a small illustration. So independently of number theory, our graph
zeta functions deserve further study. In particular, the analogous
functions for manifolds play a role in various branches of mathematical
physis. In this connection, Theorem \ref{thm:Zd} is of definite interest,
see the comments after this theorem.
\item It is also noteworthy to recall that for $s=2m$, the even positive
integers, our finite sums admit a closed form expression as a polynomial
in $n$, for example (which can be shown combinatorially in line with
the previous point), \[
\sum_{k=1}^{n-1}\frac{1}{\sin^{2}(\pi k/n)}=\frac{1}{3}n^{2}-\frac{1}{3},\]
while $\sum_{1}^{n}k^{-2m}$ does not admit such a formula. The sine
series evaluation implies, in view of (\ref{eq:zetalimit}) and (\ref{eq:Riemannzeta})
above, Euler's formulas for $\zeta(2m)$, for example $\zeta(2)=\pi^{2}/6$.
See section \ref{sec:Special-values} for more about how our asymptotical
relations imply known special values, and also references to contexts
where the finite $1/\sin$ sums are studied.
\end{enumerate}
\textbf{Higher dimensions. }For $d>1$ the torus zeta functions are
Epstein zeta functions also appearing in number theory. Some of these
are known not to satisfy the Riemann hypothesis, the statement that
all non-trivial zeros lie on one vertical line (see \cite{RH07} and
\cite{PT34}). It seems interesting to understand this difference
between $d=1$ and certain higher dimensional cases from our perspective.
Theorem \ref{thm:Zd} gives precise asymptotics in higher dimensions,
but to get even further terms in the expansion, as in Theorem \ref{thm:d=00003D1},
there are some complications, especially when trying to assemble a
nice expression, like $\zeta(s-2)$ as in Theorem \ref{thm:d=00003D1}.
Therefore this is left for future study.

\textbf{Generalized Riemann Hypothesis (GRH). }In a forthcoming sequel
about Dirichlet $L$-functions \cite{F15}, by the first-named author,
it similarily emerges that the GRH is essentially equivalent to an
expected asymptotic functional equation of the corresponding graph
$L$-function. More precisely, spectral L-functions for graphs (different
from those considered in \cite{H92} and \cite{STe00}) are introduced,
and in the case of $\mathbb{Z}/n\mathbb{Z}$, the $L$-functions completed
with suitable fudge factors, and denoted $\Lambda_{n}(s,\chi)$, satisfy\[
\lim_{n\rightarrow\infty}\left|\frac{\Lambda_{n}(s,\chi)}{\Lambda_{n}(1-s,\overline{\chi})}\right|=1,\]
for $0<Re(s)<1$ and $Im(s)\geq8$, if and only if the GRH holds (for
$s$ in the same range) for Dirichlet's $L$-function $L(s,\chi)$. 

\textbf{Zeta functions of graphs. }As recalled in the beginning, the
more standard zeta function of a graph is the one going back to a
paper by Ihara. Ihara zeta functions for infinite graphs appear in
a few places, three recent papers are \cite{D14,CJK15,LPS14}, which
contain further generalizations and where references to papers by
Grigorchuk-Zuk and Guido, Isola, and Lapidus on this topic can be
found. A two variable extension of the Ihara zeta function was introduced
by Bartholdi \cite{B99} developed out of a formula in \cite{G78}.
Zeta functions more closely related to the ones considered in the
present paper, are the spectral zeta functions of fractals in works
by Teplyaev, Lapidus and van Frankenhuijsen.

\textbf{Acknowledgement. }The second-named author thanks Jay Jorgenson,
Pär Kurlberg and Andreas Strömbergsson for valuable discussions related
to this paper. We thank Franz Lehner for suggesting the use of the
spectral measure in the calculation of the spectral zeta function
of regular trees. We thank the referee for insightful comments.

\section{Spectral zeta functions\label{sec:Spectral-zeta-functions}}

At least since Carleman \cite{Ca34} one forms a spectral zeta function
\[
\sum_{j}\frac{1}{\lambda_{j}^{s}}\]
over the set of non-zero Laplace eigenvalues, convergent for $s$
in some right half-plane. For a finite graph the elementary symmetric
functions in the eigenvalues admit a combinatorial interpretation
starting with Kirchhoff, see \emph{e.g.} \cite{CL96} for a more recent
discussion. For infinite graphs or manifolds one does at least not
a priori have such symmetric functions (since the spectrum may be
continuous or the eigenvalues are infinite in number). This is one
reason for defining spectral zeta functions, since these are symmetric,
and via transforms one can get the analytic continued intepretations
of the elementary symmetric functions, such as the (restricted) determinant.
As has been recognized at least for the determinant, the combinatorial
intepretation persists in a certain sense, see \cite{Ly10}. 

As often is the case, since Riemann, in order to define its meromorphic
continuation one writes the zeta function as the Mellin transform
of the associated theta series, or trace of the heat kernel. For this
reason and in view of that some spaces have no eigenvalues but continuous
spectrum, a case important to us in this paper, we suggest (as advocated
by Jorgenson-Lang, see for example \cite{JL12}) to start from the
heat kernel to define spectral zeta functions. Recall that the Mellin
transform of a function $f(t)$ is\[
\mathbf{M}f(s)=\int_{0}^{\infty}f(t)t^{s}\frac{dt}{t}.\]
For example when $f(t)=e^{-t}$, the transform is $\Gamma(s)$.

More precisely, for a finite or compact space $X$ we can sum over
$x_{0}$ of the unique bounded fundamental solution $K_{X}(t,x_{0},x_{0})$
of the heat equation (see for example \cite{JL12,CJK10} for more
background on this), which gives the heat trace $Tr(K_{X})$, typically
on the form $\sum e^{-\lambda t}$, and define\[
\zeta_{X}(s)=\frac{1}{\Gamma(s)}\int_{0}^{\infty}(Tr(K_{X})-1)t^{s}\frac{dt}{t}.\]

When the spectrum is discrete this formula gives back Carleman's definition
above. For a non-compact space with a heat kernel independent of the
point $x_{0}$, for example a Cayley graph of an infinite, finitely
generated group, it makes sense to take Mellin transform of $K_{X}(t,x_{0},x_{0})$
without the trace. Moreover since zero is no longer an eigenvalue
for the Laplacian acting on $L^{2}(X)$ we should no longer subtract
$1$, so the definition in this case is\[
\zeta_{X}(s)=\frac{1}{\Gamma(s)}\int_{0}^{\infty}K_{X}(t,x_{0},x_{0})t^{s}\frac{dt}{t}.\]

Let us also note that in the graph setting as shown in \cite{CJK15},
it holds that if we start with the heat kernel one may via instead
a Laplace transform obtain the Ihara zeta function and the fundamental
determinant formula.

An alternative, equivalent, definition is given by the spectral measure
$d\mu=d\mu_{x_{0},x_{0}}$ , see \cite{MW89}, \[
\zeta_{X}(s)=\int\lambda^{-s}d\mu(\lambda).\]

Here and in the next two sections we provide some examples:

\noun{Example. }For a finite torus graph defined as in the introduction
we have by calculating the eigenvalues (see for example \cite{CJK10})
\[
\zeta_{\mathbb{Z}^{d}/A\mathbb{Z}^{d}}(s)=\frac{1}{2^{2s}}\sum_{k}\frac{1}{\left(\sin^{2}(\pi k_{1}/a_{1})+...+\sin^{2}(\pi k_{d}/a_{d})\right)^{s}},\]
where the sum runs over all $0\leq k_{i}\leq a_{i}-1$ except for
all $k_{i}$s being zero.

\noun{Example. }For real tori we have again by calculating the eigenvalues
(see \cite{CJK12}) as is well known \[
\zeta_{\mathbb{R}^{d}/A\mathbb{Z}^{d}}(s)=\frac{1}{(2\pi)^{2s}}\sum_{k\in\mathbb{Z}^{d}\setminus\left\{ 0\right\} }\frac{1}{\left\Vert A^{*}k\right\Vert ^{2s}},\]
where $A^{*}=\left(A^{-1}\right)^{t}$. 

In the following sections we will discuss the zeta function of some
infinite graphs, namely the standard lattice graphs $\mathbb{Z}^{d}$.
Before that let us mention yet another example, that we again do not
think one finds in the literature.

\noun{Example. }The $(q+1)$-regular tree $T_{q+1}$ with $q\geq2$
is a fundamental infinite graph ($q=1$ corresponds to $\mathbb{Z}$
treated in the next section). Also here the spectral measure is well-known,
our reference is \cite{MW89}. Thus\[
\zeta_{T_{q+1}}(s)=\int_{-2\sqrt{q}}^{2\sqrt{q}}\frac{1}{(q+1-\lambda)^{s}}\frac{(q+1)}{2\pi}\frac{\sqrt{4q-\lambda^{2}}}{((q+1)^{2}-\lambda^{2})}d\lambda=\]

\[
=\frac{q+1}{2\pi}\int_{-2\sqrt{q}}^{2\sqrt{q}}\frac{1}{(q+1-\lambda)^{s+1}}\frac{\sqrt{4q-\lambda^{2}}}{(q+1+\lambda)}d\lambda.\]
We change variable $u=2\sqrt{q}-\lambda.$ So\[
\zeta_{T_{q+1}}(s)=\frac{q+1}{2\pi}\int_{0}^{4\sqrt{q}}\frac{1}{(q+1-2\sqrt{q}+u)^{s+1}}\frac{\sqrt{4\sqrt{q}u-u^{2}}}{(q+1+2\sqrt{q}-u)}du=\]
\[
=\frac{q+1}{2\pi}\int_{0}^{4\sqrt{q}}\frac{u^{1/2}}{(q+1-2\sqrt{q}+u)^{s+1}}\frac{\sqrt{4\sqrt{q}-u}}{(q+1+2\sqrt{q}-u)}du.\]
We change again: $u=4\sqrt{q}t,$ so \[
\zeta_{T_{q+1}}(s)=\frac{q+1}{2\pi}\int_{0}^{1}\frac{(4\sqrt{q})^{1/2}t^{1/2}}{(q+1-2\sqrt{q}+4\sqrt{q}t)^{s+1}}\frac{\sqrt{4\sqrt{q}-4\sqrt{q}t}}{(q+1+2\sqrt{q}-4\sqrt{q}t)}4\sqrt{q}dt=\]

\[
=\frac{d}{2\pi}\frac{16q}{(q+1-2\sqrt{q})^{s+1}(q+1+2\sqrt{q})}\int_{0}^{1}\frac{t^{1/2}\sqrt{1-t}}{(1-ut)^{s+1}(1-vt)}dt,\]
where $u=-4\sqrt{q}/(q+1-2\sqrt{q})$ and $v=4\sqrt{q}/(q+1+2\sqrt{q})$.
This is an Euler-type integral that Picard considered in \cite{Pi1881}
and which lead him to Appell's hypergeometric function $F_{1}$, \[
\zeta_{T_{q+1}}(s)=\frac{q+1}{2\pi}\frac{16q}{(q+1-2\sqrt{q})^{s+1}(q+1+2\sqrt{q})}\frac{\Gamma(3/2)\Gamma(3/2)}{\Gamma(3)}F_{1}(3/2,s+1,1,3;u,v).\]
Simplifing this somewhat we have proved:
\begin{thm}
For $q>1,$ the spectral zeta function of the $(q+1)$-regular tree
is\[
\zeta_{T_{q+1}}(s)=\frac{q(q+1)}{(q-1)^{2}(\sqrt{q}-1)^{2s}}F_{1}(3/2,s+1,1,3;u,v),\]
with $u=-4\sqrt{q}/(\sqrt{q}-1)^{2}$ and $v=4\sqrt{q}/(\sqrt{q}+1)^{2}$,
and where $F_{1}$ is one of Appell's hypergeometric functions. 
\end{thm}
The topic of functional relations between hypergeometric functions
is a very classical one. In spite of the many known formulas, we were
not able to derive a functional equation for $\zeta_{T_{q+1}}$ with
$s\longleftrightarrow1-s$.

\section{The spectral zeta function of the graph $\mathbb{Z}$ \label{sec:zeta of Z}}

The heat kernel of $\mathbb{Z}$ is $e^{-2t}I_{x}(2t)$ where $I_{x}$
is a Bessel function (see \cite{CJK10} and its references). Therefore

\[
\zeta_{\mathbb{Z}}(s)=\frac{1}{\Gamma(s)}\int_{0}^{\infty}e^{-2t}I_{0}(2t)t^{s}\frac{dt}{t},\]
which converges for $0<Re(s)<1/2$. It is not so clear why this function
should have a meromorphic continuation and functional equation very
similar to Riemann's zeta. 
\begin{prop}
\label{pro:beta}For $0<Re(s)<1/2$ it holds that\[
\zeta_{\mathbb{Z}}(s)=\frac{1}{4^{s}\sqrt{\pi}}\frac{\Gamma(1/2-s)}{\Gamma(1-s)}=\frac{1}{4^{s}\pi}B(1/2,1/2-s),\]
where $B$ denotes Euler's beta function. This formula provides the
meromorphic continuation of $\zeta_{\mathbb{Z}}(s)$.\end{prop}
\begin{proof}
By formula 11.4.13 in \cite{AS64}, we have\[
\mathbf{M}(e^{-t}I_{x}(t))(s)=\frac{\Gamma(s+x)\Gamma(1/2-s)}{2^{s}\pi^{1/2}\Gamma(1+x-s)},\]
valid for $Re(s)<1/2$ and $Re(s+x)>0$. This implies the first formula.
Finally, using that $\Gamma(1/2)=\sqrt{\pi}$ and the definition of
the beta function the proposition is established.
\end{proof}
We proceed to determine a functional equation for this zeta function.
Recall that \[
\Gamma(z)\Gamma(1-z)=\frac{\pi}{\sin(\pi z)}.\]
Therefore\[
2^{s}\sqrt{\pi}\zeta_{\mathbb{Z}}(s/2)=\frac{\Gamma(1-(1/2+s/2))}{\Gamma(1-s/2)}=\frac{\sin(\pi s/2)\Gamma(s/2)}{\pi}\frac{\pi}{\sin(\pi(s+1)/2)\text{\ensuremath{\Gamma}(1/2+s/2)}}\]

\[
=\tan(\pi s/2)\frac{\Gamma(1/2-(1-s)/2)}{\Gamma(1-(1-s)/2)}=2^{1-s}\sqrt{\pi}\tan(\pi s/2)\zeta_{\mathbb{Z}}((1-s)/2).\]

Hence in analogy with Riemann's case we have\[
\zeta_{\mathbb{Z}}(s/2)=2^{1-2s}\tan(\pi s/2)\zeta_{\mathbb{Z}}((1-s)/2).\]
(The passage from $s$ to $s/2$ is also the same.) If we define the
completed zeta to be \[
\xi_{\mathbb{Z}}(s)=2^{s}\cos(\pi s/2)\zeta_{\mathbb{Z}}(s/2),\]
then one verifies that the above functional equation can be written
in the familiar more symmetric form\[
\xi_{\mathbb{Z}}(s)=\xi_{\mathbb{Z}}(1-s)\]
for all $s\in\mathbb{C}.$ Moreover, note that this is an entire function
since the simple poles coming from $\Gamma$ are cancelled by the
cosine zeros and it takes real values on the critical line. We call
$\xi_{\mathbb{Z}}$ the \emph{entire completion} of $\zeta_{\mathbb{Z}}$. 

Let us determine some special values. In view of that for integers
$n\geq0$, \[
\Gamma(1/2+n)=\frac{(2n)!}{4^{n}n!}\sqrt{\pi}\]
and $\Gamma(1+n)=n!$, we have for $s=-n$, \[
\zeta_{\mathbb{Z}}(-n)=\frac{1}{4^{-n}\sqrt{\pi}}\frac{\Gamma(1/2+n)}{\Gamma(1+n)}=\frac{(2n)!}{n!n!}=\left(\begin{array}{c}
2n\\
n\end{array}\right).\]
This number equals the number of paths of length $2n$ from the origin
to itself in $\mathbb{Z}$. 

Furthermore, in a similar way for $n\geq0$, \[
\zeta_{\mathbb{Z}}(-n+1/2)=\frac{1}{4^{-n}\sqrt{\pi}}\frac{\Gamma(n)}{\Gamma(1/2+n)}=\frac{4^{2n}}{2\pi n}\frac{n!n!}{(2n)!}=\frac{4^{2n}}{2\pi n\left(\begin{array}{c}
2n\\
n\end{array}\right)}.\]

It is well-known that the gamma function is a meromorphic function
in the whole complex plane with simple poles at the negative integers
and no zeros. Note that if we pass from $s$ to $s/2$ we have that
$\zeta_{\mathbb{Z}}(s/2)$ has simple poles at the positive odd integers,
and the special values determined above appear at the even negative
numbers.

We may thus summarize:
\begin{thm}
The spectral zeta function $\zeta_{\mathbb{Z}}(s)$ can be extended
to a meromorphic function on $\mathbb{C}$ satisfying\[
\zeta_{\mathbb{Z}}(s)=\frac{1}{4^{s}\sqrt{\pi}}\frac{\Gamma(1/2-s)}{\Gamma(1-s)}.\]
It has zeros for $s=n$, $n=1,2,3...,$ and simple poles for $s=1/2+n$,
$n=0,1,2,...$ Moreover, its completion $\xi_{\mathbb{Z}}$, which
is entire, admits the functional equation\[
\xi_{\mathbb{Z}}(s)=\xi_{\mathbb{Z}}(1-s).\]
Finally we have the special values\[
\zeta_{\mathbb{Z}}(-n)=\left(\begin{array}{c}
2n\\
n\end{array}\right)\textrm{ and }\zeta_{\mathbb{Z}}(-n+1/2)=\frac{4^{2n}}{2\pi n\left(\begin{array}{c}
2n\\
n\end{array}\right)},\]
where $n\geq0$ is an integer. 
\end{thm}

\section{The spectral zeta function of the lattice graphs $\mathbb{Z}^{d}$
\label{sec:zeta of Zd}}

The heat kernel on $\mathbb{Z}^{d}$ is the product of heat kernels
on $\mathbb{Z}$ and this gives that

\[
\zeta_{\mathbb{Z}^{d}}(s)=\frac{1}{\Gamma(s)}\mbox{\ensuremath{\int_{0}^{\infty}e^{-2dt}I_{0}}(2t\ensuremath{)^{d}t^{s}\frac{dt}{t}},}\]
which converges for $0<Re(s)<d/2$. For $d=2$ taking instead the
equivalent definition with the spectral measure, the spectral zeta
function is a variant of the Selberg integral with two variables. 

The integrals like \[
\int_{0}^{\infty}e^{-zt}I_{0}(2t)^{d}t^{s}\frac{dt}{t},\]
and more general ones, have been studied by Saxena in \cite{Sa66},
see also the discussion in \cite[sect. 9.4]{SK85}. For $Re(z)>2d$
and $Re(s)>0$ one has\[
\int_{0}^{\infty}e^{-zt}I_{0}(2t)^{d}t^{s}\frac{dt}{t}=\frac{2^{s-1}}{\sqrt{\pi}}z^{-s+1/2}\Gamma(\frac{s+1}{2})F_{C}^{(d)}(s/2,(s+1)/2;1,1,...,1;4/z^{2},4/z^{2},...,4/z^{2}),\]
where $F_{C}^{(d)}$ is one of the Lauricella hypergeometric functions
in $d$ variables \cite{Ex76}. The condition $Re(z)>2d$ can presumably
be relaxed by the principle of analytic continuation giving up the
multiple series definition of $F_{C}^{(d)}$. This point is discussed
in \cite{SE79}. Formally we would then have that \[
\zeta_{\mathbb{Z}^{d}}(s)=\frac{d^{-s+1/2}}{\sqrt{2\pi}}\frac{\Gamma((s+1)/2)}{\Gamma(s)}F_{C}^{(d)}(s/2,(s+1)/2;1,1,...,1;1/d^{2},1/d^{2},...,1/d^{2}),\]

which is rather suggestive as far as functional relations go. It is
however not clear at present time that for $d>1$ there is a relation
as nice as the functional equation in the case $d=1.$ Related to
this, it is remarked in \cite[p. 49]{Ex76} that no integral representation
of Euler type has been found for $F_{C}$. We note that if one instead
of the heat kernel start with the spectral measure in defining $\zeta_{\mathbb{Z}^{d}}(s)$,
we do get such an integral representation, at least for special parameters.
This aspect is left for future investigation.

We will now provide an independent and direct meromorphic continuation
of these functions. To do this, we take advantage of the heat kernel
definition of the zeta function. Fix a dimension $d\geq1$. Recall
that on the one hand there are explicit positive non-zero coefficients
$a_{n}$ such that\[
e^{-2dt}I_{0}(2t)^{d}=\sum_{n\geq0}a_{n}t^{n}\]
which converges for every positive $t$, and on the other hand we
similarily have an expansion at infinity, \[
e^{-2dt}I_{0}(2t)^{d}=\sum_{n=0}^{N-1}b_{n}t^{-n-d/2}+O(t^{-N-d/2})\]
as $t\rightarrow\infty$ for any integer $N>0$. 

Therefore we write\[
\int_{0}^{\infty}e^{-2dt}I_{0}(2t)^{d}t^{s-1}dt=\int_{0}^{1}\sum_{n=0}^{N-1}a_{n}t^{n}t^{s-1}dt+\int_{0}^{1}\sum_{n\geq N}a_{n}t^{n}t^{s-1}dt+\]

\[
+\int_{1}^{\infty}\left(e^{-2dt}I_{0}(2t)^{d}-\sum_{n=0}^{N-1}b_{n}t^{-n-d/2}\right)t^{s-1}dt+\int_{1}^{\infty}\sum_{n=0}^{N-1}b_{n}t^{-n-d/2}t^{s-1}dt=\]

\[
=\sum_{n=0}^{N-1}\frac{a_{n}}{s+n}+\sum_{n=0}^{N-1}\frac{b_{n}}{s-(n+d/2)}+\int_{0}^{1}O(t^{N})t^{s-1}dt+\int_{1}^{\infty}O(t^{-N-d/2})t^{s-1}dt.\]
This last expression defines a meromorphic function in the region
$-N<Re(s)<N+d/2$, with simple poles at $s=-n$ and $s=n+d/2$. 

The spectral zeta function $\zeta_{\mathbb{Z}^{d}}(s)$ is the above
integral divided by $\Gamma(s)$. In view of that the entire function
$1/\Gamma(s)$ has zeros at the non-positive integers, this will cancel
the simple poles at $s=-n$. Since we can take $N$ as large as we
want we obtain in this way the meromorphic continuation of $\zeta_{\mathbb{Z}^{d}}(s)$.
Moreover, thanks to that the coefficients $b_{n}$ are non-zero we
have established: 
\begin{prop}
The function $\zeta_{\mathbb{Z}^{d}}(s)$ admits a meromorphic continuation
to the whole complex plane with simple poles at the points $s=n+d/2$
with $n\geq0$.
\end{prop}
It is natural to wonder whether this function also for $d>1$ can
be completed like in the case $d=1$ giving an entire function with
functional relation $\xi_{\mathbb{Z}^{d}}(1-s)=\xi_{\mathbb{Z}^{d}}(s)$.
Indeed, more generally we find the question interesting for which
graph, finite or infinite, the zeta functions have a functional relation
in some way analogous to the classical type of functional equations.

Finally we point out a non-trivial special value that we derive in
a later section: \[
\zeta_{\mathbb{Z}^{d}}(0)=1.\]

\section{Asymptotics of the zeta functions of torus graphs\label{sec:Asymptotics}}

We consider a sequence of torus graphs $\mathbb{Z}^{d}\mathrm{/}A_{n}\mathbb{Z}^{d}$
indexed by $n$ and where the matrices $A_{n}$ are diagonal with
entries $a_{i}n$, with integers $a_{i}>0.$ (A more general setting
could be considered (\emph{cf. }\cite{CJK12}) but it will not be
important to us in the present context.) We denote by $\zeta_{n}$
the corresponding zeta function defined as in the previous section.
We let the matrix $A$ be the diagonal matrix with entries $a_{i}$.
In this section we take advantage of the theory developed in \cite{CJK10}
without recalling the proofs which would take numerous pages. 

Following \cite{CJK10} we have\[
\theta_{n}(t):=\sum_{m}e^{-\lambda_{m}t}=\det(A_{n})\sum_{k\in\mathbb{Z}\mathrm{^{d}}}\prod_{1\leq j\leq d}e^{-2t}I_{a_{j}nk_{j}}(2t),\]
where $\lambda_{m}$ denotes the Laplace eigenvalues. From the left
hand side it is clear that this function is entire. Let\[
\theta_{A}(t)=\sum_{\lambda}e^{-\lambda t},\]
where the sum is over the eigenvalues of the torus $\mathbb{R}^{d}/A\mathbb{Z}^{d}$.
The meromorphic continuation of the corresponding spectral zeta function
is, as is well-known (see \emph{e.g.} \cite{CJK10}),\[
\zeta_{\mathbb{R}^{d}/A\mathbb{Z}^{d}}(s)=\frac{1}{\Gamma(s)}\int_{1}^{\infty}\left(\theta_{A}(t)-1\right)t^{s}\frac{dt}{t}+\frac{1}{\Gamma(s)}\int_{0}^{1}\left(\theta_{A}(t)-\det A(4\pi t)^{-d/2}\right)t^{s}\frac{dt}{t}+\]

\[
+\frac{(4\pi)^{-d/2}\det A}{\Gamma(s)(s-d/2)}-\frac{1}{s\Gamma(s)}.\]

Recall the asymptotics for the I-Bessel functions: \[
I_{n}(x)=\frac{e^{x}}{\sqrt{2\pi x}}\left(1-\frac{4n^{2}-1}{8x}+O(x^{-2})\right)\]
as $x\rightarrow\infty.$

For $0<Re(s)<d/2$ we may write\[
\Gamma(s)\zeta_{n}(s)=\int_{0}^{\infty}\left(\theta_{n}(t)-1\right)t^{s}\frac{dt}{t}=n^{2s}\int_{0}^{\infty}\left(\theta_{n}(n^{2}t)-1\right)t^{s}\frac{dt}{t}.\]
We decompose the integral on the right and let $n\rightarrow\infty$,
the first piece being\[
S_{1}(n):=\int_{1}^{\infty}\left(\theta_{n}(n^{2}t)-1\right)t^{s}\frac{dt}{t}\rightarrow\int_{1}^{\infty}\left(\theta_{A}(t)-1\right)t^{s}\frac{dt}{t}\]
for every $s\in\mathbb{C}$ as $n\rightarrow\infty$. The convergence
is proved in \cite{CJK10}. The second piece is for $Re(s)>-n$, \[
S_{2}(n):=\int_{0}^{1}\left(\theta_{n}(n^{2}t)-\det A_{n}e^{-2dn^{2}t}I_{0}(2n^{2}t)^{d}\right)t^{s}\frac{dt}{t}\rightarrow\int_{0}^{1}\left(\theta_{A}(t)-\det A(4\pi t)^{-d/2}\right)t^{s}\frac{dt}{t},\]
as $n\rightarrow\infty$ which is proved in \cite{CJK10}. 

What remains is now the third piece\[
S_{3}(n):=\int_{0}^{1}\left(\det A_{n}e^{-2dn^{2}t}I_{0}(2n^{2}t)^{d}-1\right)t^{s}\frac{dt}{t}=n^{-2s}\int_{0}^{n^{2}}\left(\det A_{n}e^{-2dt}I_{0}(2t)^{d}-1\right)t^{s}\frac{dt}{t}.\]
This we write as follows\[
S_{3}(n)=\left(\det A_{n}\int_{0}^{\infty}e^{-2dt}I_{0}(2t)^{d}t^{s}\frac{dt}{t}-\det A_{n}\int_{n^{2}}^{\infty}e^{-2dt}I_{0}(2t)^{d}t^{s}\frac{dt}{t}-\int_{0}^{n^{2}}t^{s}\frac{dt}{t}\right)n^{-2s}.\]
The first integral is the spectral zeta of $\mathbb{Z}^{d}$ times
$\Gamma(s)$ and the last integral is \[
\int_{0}^{n^{2}}t^{s}\frac{dt}{t}=\frac{n^{2s}}{s}.\]
We continue with the middle integral here: \[
\int_{n^{2}}^{\infty}e^{-2dt}I_{0}(2t)^{d}t^{s}\frac{dt}{t}=\int_{n^{2}}^{\infty}\left(e^{-2dt}I_{0}(2t)^{d}-(4\pi t)^{-d/2}\right)t^{s}\frac{dt}{t}+\int_{n^{2}}^{\infty}(4\pi t)^{-d/2}t^{s}\frac{dt}{t},\]
hence\[
\int_{n^{2}}^{\infty}e^{-2dt}I_{0}(2t)^{d}t^{s}\frac{dt}{t}=\int_{n^{2}}^{\infty}\left(e^{-2dt}I_{0}(2t)^{d}-(4\pi t)^{-d/2}\right)t^{s}\frac{dt}{t}-(4\pi)^{-d/2}\frac{n^{2s-d}}{s-d/2}.\]
We denote\[
S_{rest}(n)=\int_{n^{2}}^{\infty}\left(e^{-2dt}I_{0}(2t)^{d}-(4\pi t)^{-d/2}\right)t^{s}\frac{dt}{t},\]
which is a convergent integral for $Re(s)<d/2+1$ in view of the asymptotics
for $I_{0}(t)$. Notice also that for fixed $s$ with $Re(s)<d/2+1$
the integral is of order $n^{2s-2-d}$ as $n\rightarrow\infty.$

Taken all together we have\[
n^{-2s}\zeta_{n}(s)=\frac{1}{\Gamma(s)}S_{1}(n)+\frac{1}{\Gamma(s)}S_{2}(n)-\frac{1}{s\Gamma(s)}+(4\pi)^{-d/2}\frac{\det A}{\Gamma(s)(s-d/2)}+\]

\[
+n^{d-2s}\det A\,\zeta_{\mathbb{Z}^{d}}(s)-n^{d-2s}\frac{\det A}{\Gamma(s)}S_{rest}(n).\]
This is valid for all $s$ in the intersection of where $\zeta_{\mathbb{Z}^{d}}(s)$
is defined, $-n<Re(s)<d/2+1$, and $s\neq d/2.$ As remarked above
coming from \cite{CJK10} as $n\rightarrow\infty$ the first four
terms combines to give $\zeta_{\mathbb{R}^{d}/A\mathbb{Z}^{d}}(s)$.
This means that we have in particular proved Theorem \ref{thm:Zd}.

\section{The one dimensional case\label{sec:dimension one}}

We now specialize to $d=1$ and $A_{n}=n.$ In this case recall that\[
\zeta_{n}(s)=\zeta_{\mathbb{Z\mathrm{/}}n\mathbb{Z}}(s)=\frac{1}{4^{s}}\sum_{k=1}^{n-1}\frac{1}{\sin(\pi k/n)^{2s}}\]
and\[
\zeta_{\mathbb{R\mathrm{/\mathbb{Z}}}}(s)=2(2\pi)^{-2s}\zeta(2s),\]

where $\zeta$ is the Riemann zeta function. Moreover, \[
\zeta_{\mathbb{Z}}(s)=\frac{1}{4^{s}\sqrt{\pi}}\frac{\Gamma(1/2-s)}{\Gamma(1-s)}.\]
In view of the previous section the first part of Theorem \ref{thm:d=00003D1}
is established. Let us remark that this can also be viewed as a special
case of Gauss-Chebyshev quadrature but with a more precise error term.

With more work one can also find the next term in the asymptotic expansion
in the critical strip. This can be achieved with some more detailed
analysis, in particular of Proposition 4.7 in \cite{CJK10} and an
application of Poisson summation. For the purpose of the present discussion
we only need to look at the more precise asymptotics in the critical
strip and here for $d=1$ there is an alternative approach available
by using a non-standard version of the Euler-Maclaurin formula established
in \cite{Si04}. The asymptotics is:\[
\sum_{k=1}^{n-1}\frac{1}{\sin(\pi k/n)^{s}}=\frac{1}{\sqrt{\pi}}\frac{\Gamma(1/2-s/2)}{\Gamma(1-s/2)}n+2\pi^{-s}\zeta(s)n^{s}+\frac{s}{3}\pi^{2-s}\zeta(s-2)n^{s-2}+o(n^{s-2})\]
where $0<Re(s)<1$ as $n\rightarrow\infty.$ This is the second statement
in Theorem \ref{thm:d=00003D1}.

\noun{Example: }Although we did not verify this asymptotics outside
of the critical strip, it may nevertheless be convincing to specialize
to $s=2$ , we then would have\[
\frac{1}{3}n^{2}-\frac{1}{3}=\frac{1}{\sqrt{\pi}}0\cdot n+2\pi^{-2}\zeta(2)n^{2}+\frac{2}{3}\zeta(0)+o(1),\]
which confirms the values $\zeta(0)=-1/2$ and $\zeta(2)=\pi^{2}/6$.
As remarked in the introduction, from \cite{CJK10}, the value of
$\zeta(2)$ can also be derived via \[
\frac{2}{\pi^{2}}\zeta(2)=\lim_{n\rightarrow\infty}\frac{1}{n^{2}}\left(\frac{1}{3}n^{2}-\frac{1}{3}\right).\]

\section{Special values\label{sec:Special-values}}

\subsection{The case of $s=0$}

Setting $s=0$ in Theorem \ref{thm:Zd} we clearly have\[
\det A\: n^{d}-1=\zeta_{\mathbb{Z}^{d}}(0)\det A\: n^{d}+\zeta_{\mathbb{R}^{d}/A\mathbb{Z}^{d}}(0)+o(1),\]
which implies that $\zeta_{\mathbb{Z}^{d}}(0)=1$, and that $\zeta_{\mathbb{R}^{d}/A\mathbb{Z}^{d}}(0)=-1$,
which is a known special value of Epstein zeta functions.

\subsection{The case of $s$ being negative integers\label{sub:special-combinatorics}}

Let us now recall some known results about the sums: \[
\sum_{k=1}^{n-1}\frac{1}{\sin(\pi k/n)^{s}}\]
 for special $s$. We begin with a simple calculation (see for example
\cite[Lemma 3.5]{BM10}) namely that for integers $0<m<n$ \[
\sum_{k=1}^{n-1}\sin^{2m}(\pi k/n)=\frac{n}{4^{m}}\left(\begin{array}{c}
2m\\
m\end{array}\right).\]
In view of the asymptotics in Theorem \ref{thm:d=00003D1} this immediately
imply that $\zeta(-2m)=0$, the so-called trivial zeros of Riemann's
zeta function. It also verifies with the special values of $\zeta_{\mathbb{Z}}$
stated in section \ref{sec:zeta of Z}. There is a probabilistic interpretation
for this: when the number of steps $m$ is smaller than $n$, the
random walker cannot tell the difference between the graphs $\mathbb{Z}$
and $\mathbb{Z}/n\mathbb{Z}$.

Conversely, for $s$ being an odd negative integer our asymptotic
formula gives information about the sine sum which is somewhat more
complicated in this case, as the fact that $\zeta$ does not vanish
implies. For low exponent $m$ one can find formulas in \cite{GR07},
the simplest one being\[
\sum_{k=1}^{n-1}\sin(k\pi/n)=\cot(\pi/2n).\]

\subsection{The case of $s$ being even positive integers}

In view of the elementary equality \[
\frac{1}{\sin^{2}x}=1+\cot^{2}x,\]
one sees that for positive integers $a$,\[
\sum_{k=1}^{n-1}\frac{1}{\sin^{2a}(\pi k/n)}\]
can be expresed in terms of higher Dedekind sums considered by Zagier
\cite{Z73}. There is also a literature more specialized on this type
of finite sums which can be evaluated with a closed form expression
already mentioned in the introduction (see \cite{CM99,BY02}):\[
\sum_{k=1}^{n-1}\frac{1}{\sin^{2a}(\pi k/n)}=-\frac{1}{2}\sum_{m=0}^{2a}\frac{(-4)^{a}}{n^{m}}\left(\begin{array}{c}
2a+1\\
m+1\end{array}\right)\times\]
\[
\times\sum_{k=0}^{m+1}(-1)^{k}\left(\begin{array}{c}
m+1\\
k\end{array}\right)\frac{m+1-2k}{m+1}\left(\begin{array}{c}
a+kn+(m-1)/2\\
2a+m\end{array}\right).\]
These sums apparently arose in physics in Dowker's work and in mathematical
work of Verlinde (see \cite{CS12}). The first order asymptotics is
known to be \[
\sum_{k=1}^{n-1}\sin^{-2m}(\pi k/n)\sim(-1)^{m+1}(2n)^{2m}\frac{B_{2m}}{(2m)!},\]
where $m$ is a positive integer, see for example \cite{BY02,CS12}
and their references. As explained in the introduction these evaluations
together with the asymptotics formulated in the introduction re-proves
Euler's celebrated calculations of $\zeta(2m)$.

At $s=1$, the point where our asymptotic expansion does not apply
because of the pole of $\zeta$, one has (see \cite[p. 460]{He77}
attributed to J. Waldvogel)\[
\zeta_{\mathbb{Z\mathrm{/}}n\mathbb{Z}}(1)=\frac{2n}{\pi}\left(\log(2n/\pi)-\gamma\right)+O(1),\]
where as usual $\gamma$ is Euler's constant.

\subsection{Further special values}

Recall the values $\Gamma(1/2)=\sqrt{\pi}$, $\Gamma'(1)=-\gamma$
and $\Gamma'(1/2)=-\gamma\sqrt{\pi}-\log4$, or in the logarithmic
derivative, the psi-function, $\psi(1)=-\gamma$ and $\psi(1/2)=-\gamma-2\log2$.
We differentiate $\zeta_{\mathbb{Z}}(s)$ which gives \[
\zeta_{\mathbb{Z}}'(s)=\zeta_{\mathbb{Z}}(s)\left(-2\log2-\psi(1/2-s)+\psi(1-s)\right).\]
Setting $s=0$ and inserting the special values mentioned we see that\[
\zeta_{\mathbb{Z}}'(0)=0.\]
This value has the interpretation of being the tree entropy of $\mathbb{Z}$,
which is the exponential growth rate of spanning trees of subgraphs
converging to $\mathbb{Z}$, see e.g. \cite{DD88,Ly10,CJK10}, studied
via the Fuglede-Kadison determinant of the Laplacian. This has also
a role in the theory of operator algebras, but in any case it is not
evaluated in this way in the literature. Of course one could in our
way compute other special values of $\zeta_{\mathbb{Z}}'$. For example,
at positive integers and half-integers this function has zeros and
poles, respectively, and at negative integers we have for integers
$n>0$ the following:
\begin{prop}
It holds that\[
\zeta_{\mathbb{Z}}'(-n)=\left(\begin{array}{c}
2n\\
n\end{array}\right)\left(1+\frac{1}{2}+...+\frac{1}{n}-2\left(1+\frac{1}{3}+...+\frac{1}{2n-1}\right)\right)\]
and\[
\zeta_{\mathbb{Z}}'(-n+1/2)=\frac{4^{2n}}{2\pi n\left(\begin{array}{c}
2n\\
n\end{array}\right)}\left(-4\log4-1-\frac{1}{2}-\frac{1}{3}...-\frac{1}{n-1}+2\left(1+\frac{1}{3}+...+\frac{1}{2n-1}\right)\right).\]

\end{prop}
We remark that this section concerned mostly $d=1$, we have not investigated
the case of higher dimensions.

\subsection{No real zero in the crticial strip}

The non-vanishing of number theoretic zeta functions on the real line
in the critical strip is of importance, see e.g. \cite{SC67}. We
outline one possible strategy for this problem in general from our
asymptotics. We treat here only Rieman's zeta function for illustration,
in this case there are however other more elementary arguments available.

Already the beginning of this section indicates that certain Epstein
zeta functions have a tendency to be negative on the real line in
the critical strip. It is as if the number of terms in the finite
graph zeta is not enough to account for the limit graph zeta function,
leaving the relevant Epstein zeta function negative. 

The function $\sin(\pi x)^{-s}$ for $0<s<1$ is positive, convex
and symmetric around $x=1/2.$ The graph zeta funciton in question,
$\zeta_{\mathbb{Z}}(s)$ is via a change of variables \[
\int_{0}^{1}\sin^{-s}(\pi x)dx.\]
If we compare this with the sum, using the symmetry, we have for odd
$n$\[
2\left(\frac{1}{n}\sum_{k=1}^{(n-1)/2}\sin^{-s}(\pi k/n)-\int_{0}^{1/2}\sin^{-s}(\pi x)dx\right)=\frac{2\zeta(s)}{\pi^{s}}n^{s-1}+o(n^{s-1}).\]
If we interpret the sum as the Riemann sum of the integral (with not
enough terms) the integral can be thought of as always lying above
the rectangles. Ignoring all but one rectangle then gives\[
\frac{1}{n}\sum_{k=1}^{(n-1)/2}\sin^{-s}(\pi k/n)-\int_{0}^{1/2}\sin^{-s}(\pi x)dx<\]
\[
<\frac{1}{n}\sin^{-s}(\pi/n)-\int_{0}^{1/n}\sin^{-s}(\pi x)dx=\frac{1}{n}\frac{n^{s}}{\pi^{2}}-\frac{1}{\pi^{s}}\int_{0}^{1/n}x^{-s}dx+o(n^{s-1})=\]
\[
=\frac{n^{s-1}}{\pi^{s}}\left(1-\frac{1}{1-s}\right)+o(n^{s-1}).\]
This shows by letting $n$ go to infinity that \[
\zeta(s)\leq-\frac{s}{1-s}<0,\]
which is consistent with numerics, for example, $\zeta(1/2)=-1.460...<-1$. 

As with several other aspects of this paper, we leave higher dimensions
to future study.

\section{Approximative functional equations\label{sec:Approximative-functional-equations}}

It is natural to wonder about to what extent $\zeta_{\mathbb{Z}/n\mathbb{Z}}$
has a functional equation. In view of our asymptotics and the, in
this context crucial, relation $\xi_{\mathbb{Z}}(s)=\xi_{\mathbb{Z}}(1-s)$,
one could expect at least an asymptotic version. Indeed, we start
by completing the finite torus zeta functions as $\xi_{\mathbb{Z}/n\mathbb{Z}}(s):=2^{s}\cos(\pi s/2)\zeta_{\mathbb{Z}/n\mathbb{Z}}(s/2)$,
and multiply the asymptotics at $s$ in the critical strip with the
corresponding fudge factors, and do the similar thing for the corresponding
formula at $1-s$. After that, we subtract the two expressions, the
one at $s$ with the one at $1-s$, and obtain after further calculations,
notably using $\xi_{\mathbb{Z}}(s)=\xi_{\mathbb{Z}}(1-s)$:\[
\xi_{\mathbb{Z}/n\mathbb{Z}}(s)-\xi_{\mathbb{Z}/n\mathbb{Z}}(1-s)=X(s)n^{s}-X(1-s)n^{1-s}+\]
\[
-\frac{s}{6}X(s-2)n^{s-2}+\frac{1-s}{6}X((1-s)-2)n^{(1-s)-2}+o(n^{a}),\]
where $a=\max\left\{ Re(s)-2,-1-Re(s)\right\} $ and $X(s)=2\pi^{-s}\cos(\pi s/2)\zeta(s)$.
Thus:
\begin{cor}
The Riemann zeta function has a zero at $s$ in the critical strip
iff \[
\lim_{n\rightarrow\infty}\left(\xi_{\mathbb{Z}/n\mathbb{Z}}(1-s)-\xi_{\mathbb{Z}/n\mathbb{Z}}(s)\right)=0\]
as $n\rightarrow\infty$, unless $s=1/2$. In any case, for all $s$
in the critical strip\[
\lim_{n\rightarrow\infty}\frac{1}{n}\left(\xi_{\mathbb{Z}/n\mathbb{Z}}(1-s)-\xi_{\mathbb{Z}/n\mathbb{Z}}(s)\right)=0\]

\end{cor}
As is well known there is a very useful approximative functional equation
for $\zeta(s)$, sometimes called the Riemann-Siegel formula, which
states that \[
\zeta(s)=\sum_{k=1}^{n}\frac{1}{k^{s}}+\pi^{s-1/2}\frac{\Gamma((1-s)/2)}{\Gamma(s/2)}\sum_{k=1}^{m}\frac{1}{k^{1-s}}+R_{n,m}(s),\]
where $R_{m,n}$ is the error term. Notice that the two partial Dirichlet
series here have the same sign, which is a different feature from
the formulas above. A question here is what functional equations prevail
in higher dimension $d$.

\section{The Riemann hypothesis\label{sec:The-Riemann-hypothesis}}

From the asymptotics given in the theorems above there is a straightforward
reformulation of the Riemann hypothesis in terms of the asymptotical
behaviour of \[
\sum_{k=1}^{n-1}\frac{1}{\sin(\pi k/n)^{s}}\]
as $n\rightarrow\infty$ as a function of $s$. It turns out however,
that there is a more unexpected, nontrivial, and, what we think, more
interesting equivalence with the Riemann hypothesis. 

To show this we begin from the second asymptotical formula in Theorem
\ref{thm:d=00003D1}:\[
\sum_{k=1}^{n-1}\frac{1}{\sin(\pi k/n)^{s}}=\frac{1}{\sqrt{\pi}}\frac{\Gamma(1/2-s/2)}{\Gamma(1-s/2)}n+2\pi^{-s}\zeta(s)n^{s}+\frac{s}{3}\pi^{2-s}\zeta(s-2)n^{s-2}+o(n^{s-2})\]
for $0<Re(s)<1$ as $n\rightarrow\infty.$

Let\[
h_{n}(s)=(4\pi)^{s/2}\Gamma(s/2)n^{-s}\left(\zeta_{\mathbb{Z}/n\mathbb{Z}}(s/2)-n\zeta_{\mathbb{Z}}(s/2)\right)=\]
\[
=\pi^{s/2}\Gamma(s/2)n^{-s}\left(\sum_{k=1}^{n-1}\frac{1}{\sin(\pi k/n)^{s}}-\frac{1}{\sqrt{\pi}}\frac{\Gamma(1/2-s/2)}{\Gamma(1-s/2)}n\right).\]
Using the completed Riemann zeta function $\xi(s):=\pi^{-s/2}\Gamma(s/2)\zeta(s)$
the above asymptotics can be restated as\[
h_{n}(s)=2\xi(s)+\alpha(s)n^{-2}+o(n^{-2}),\]
where $\alpha(s):=\frac{s}{3}\pi^{2-s/2}\Gamma(s/2)\zeta(s-2).$ 

From this asymptotics and in view of $\xi(1-s)=\xi(s)$ we conclude
immediately:
\begin{prop}
Let $s\in\mathbb{C}$ with $0<Re(s)<1$ and $\zeta(s)\neq0$. Then
$h_{n}(1-s)\sim h_{n}(s)$ in the sense that\[
\lim_{n\rightarrow\infty}\frac{h_{n}(1-s)}{h_{n}(s)}=1.\]

\end{prop}
We now conjecture that a weakened version of this asymptotic functional
relation is valid even at zeta zeros:
\begin{conjecture*}
Let $s\in\mathbb{C}$ with $0<Re(s)<1.$ Then\[
\lim_{n\rightarrow\infty}\left|\frac{h_{n}(1-s)}{h_{n}(s)}\right|=1.\]

\end{conjecture*}
From now on we will prove that this is equivalent to the Riemann hypothesis:
\begin{thm*}
The conjecture above is equivalent to the statement that all non-trivial
zeros of $\zeta$ have real part $1/2.$
\end{thm*}
We begin the proof with a simple observation:
\begin{lem}
\label{lem:functionalequation}Suppose $\zeta(s)=0.$ Then the asymptotic
relation\[
\lim_{n\rightarrow\infty}\left|\frac{h_{n}(1-s)}{h_{n}(s)}\right|=1\]
 is equivalent to $\left|\alpha(1-s)\right|=\left|\alpha(s)\right|$. 
\end{lem}
Next we have:
\begin{lem}
The equation $\left|\alpha(1-s)\right|=\left|\alpha(s)\right|$ holds
for all $s$ on the critical line $Re(s)=1/2$.\end{lem}
\begin{proof}
Recall that \[
\alpha(s)=\frac{s}{3}\pi^{2-s/2}\Gamma(s/2)\zeta(s-2).\]
Since $\zeta(\overline{s})=\overline{\zeta(s)}$ and $\Gamma(\overline{s})=\overline{\Gamma(s)}$,
we have that $\alpha(\overline{s})=\overline{\alpha(s)}$. Therefore
if $s=1/2+it$, then \[
\alpha(1-s)=\alpha(1-1/2-it)=\alpha(\overline{1/2+it})=\overline{\alpha(s)},\]
which implies the lemma.
\end{proof}
Note that using $\xi((1-s)-2)=\xi(s+2)$ and Euler's reflection formula
$\Gamma(z)\Gamma(1-z)=\pi/\sin(\pi z)$ we have \[
\left|\frac{\alpha(1-s)}{\alpha(s)}\right|=\left|\frac{\frac{(s-1)(s+1)}{6}\pi\pi^{-(s+2)/2}\Gamma((s+2)/2)\zeta(s+2)}{\frac{s(s-2)}{6}\pi\pi^{-(s-2)/2}\Gamma((s-2)/2)\zeta(s-2)}\right|=\left|\frac{\zeta(s+2)(s-1)(s+1)}{\zeta(s-2)4\pi^{2}}\right|.\]
As a consequence $\left|\alpha(1-s)\right|=\left|\alpha(s)\right|$
is equivalent to\[
\left|\frac{\zeta(s+2)}{\zeta(s-2)}\right|=\frac{4\pi^{2}}{\left|s^{2}-1\right|}.\]

We will study the right and left hand sides as functions of $\sigma$,
in the interval $0<\sigma<1$, with $s=\sigma+it$ and $t>0$ fixed.
In view of that \[
\frac{1}{\left|s^{2}-1\right|^{2}}=\frac{1}{\sigma^{4}+2\sigma^{2}+(t^{2}-1)^{2}},\]
we see that the right hand side is strictly decreasing in $\sigma$.
On the other hand we have the following:
\begin{lem}
\label{lem:Riemann}Let $s=\sigma+it$, with $t$ fixed such that
$\left|t\right|>26$. Then the function\[
\left|\frac{\zeta(s+2)}{\zeta(s-2)}\right|\]
is strictly increasing in $0<\sigma<1$.\end{lem}
\begin{proof}
As remarked in \cite{MSZ14}, for a homolorphic function $f$, a simple
calculation, using the Cauchy-Riemann equation, leads to \[
Re(f'(s)/f(s))=\frac{1}{\left|f(s)\right|}\frac{\partial\left|f(s)\right|}{\partial\sigma},\]
in any domain where $f(z)\neq0$. This implies that for $\left|f\right|$
to be increasing in $\sigma$ we should show that the real part of
its logarithmic derivative is positive.

We begin with one of the two terms in the logarithmic derivative of
$\zeta(s+2)/\zeta(s-2)$: \[
Re(\zeta'(s+2)/\zeta(s+2))=-Re(\sum_{n\geq1}\Lambda(n)n^{-s-2})=-\sum_{n\geq1}\Lambda(n)n^{-\sigma-2}\cos(t\log n),\]
where $\Lambda(n)$is the von Mangoldt function. So\[
\left|Re(\zeta'(s+2)/\zeta(s+2))\right|\leq\sum_{n\geq1}\Lambda(n)n^{-2}=-\frac{\zeta'(2)}{\zeta(2)}=\gamma+\log(2\pi)-12\log A<0.57,\]
by known numerics. We are therefore left to show that the other term
\[
Re(-\zeta'(s-2)/\zeta(s-2))\geq0.57.\]

On the one hand, following the literature, see \cite{L99,SD10,MSZ14},
from the Mittag-Leffler expansion we have\[
\frac{\tilde{\xi}'(s)}{\tilde{\xi}(s)}=\sum_{\rho}\frac{1}{s-\rho}\]
where the sum is taken over the zeros which all lie in the critical
strip. (The function $\tilde{\xi}(s)$ is defined by $\tilde{\xi}(s)=(s-1)\Gamma(1+s/2)\pi^{-s/2}\zeta(s)$.)
This implies by a simple termwise calculation (\cite{MSZ14}) that
since $s-2$ is to the left of the critical strip, we have $Re(\tilde{\xi}'(s-2)/\tilde{\xi}(s-2))<0$
in the interval $0<\sigma<1$. On the other hand\[
0>Re(\tilde{\xi}'(s-2)/\tilde{\xi}(s-2))=Re(1/(s-3))+\frac{1}{2}Re(\psi(s/2))-\frac{1}{2}\log\pi+Re(\zeta'(s-2)/\zeta(s-2)),\]
where $\psi$ is the logarithmic derivative of the gamma function.
We estimate \[
Re(1/(s-3))=\frac{\sigma-3}{(\sigma-3)^{2}+t^{2}}>\frac{-3}{4+t^{2}}>-\frac{3}{4+144}>-0.03\]
and $-\log\pi>-1.2$. Hence\[
Re(-\zeta'(s-2)/\zeta(s-2))>-0.7+Re(\psi(s/2))/2.\]
The last thing to do is to estimate the psi-function. Following \cite{MSZ14},
we have using Stirling's formula for $\psi$, \[
Re(\psi(s))=\log\left|s\right|-\frac{\sigma}{2\left|s\right|^{2}}+Re(R(s)),\]
where $\left|R(s)\right|\leq\sqrt{2}/(6\left|s\right|^{2})$. This
is valid for any $s=\sigma+it$ in the critical strip. We observe
that\[
-\frac{\sigma}{2\left|s\right|^{2}}\geq-\frac{1}{2t^{2}}\]
so \[
Re(\psi(s/2))\geq\log\frac{\left|t\right|}{2}-\frac{2}{t^{2}}-\frac{2\sqrt{2}}{3t^{2}}\geq2.56\]
if $\left|t\right|\geq26.$ This completes the proof.
\end{proof}
Note that by numerics one can see that the lemma does not hold for
small $t$. The lemma implies that the left and right hand sides can
be equal only once for a fixed $t$, and this occurs at $Re(s)=1/2$
as shown above. We summarize this in the following statement which
concerns just the Riemann zeta function: 
\begin{prop}
For $s\in\mathbb{C}$ with $0<Re(s)<1,$ with $\left|Im(s)\right|>26,$
the equality $\left|\alpha(1-s)\right|=\left|\alpha(s)\right|$ holds
if and only if $Re(s)=1/2$. 
\end{prop}
Therefore, since it is known that the Riemann zeta zeros in the critical
strip having imaginary part less than 26 in absolute value all lie
on the critcal line and in view of Lemma \ref{lem:functionalequation},
the equivalence between the graph zeta functional equation and the
Riemann hypothesis is established.

\noindent Fabien Friedli 

\noindent Section de mathématiques \\
Université de Genève

\noindent 2-4 Rue du Lièvre\\
Case Postale 64

\noindent 1211 Genève 4, Suisse 

\noindent e-mail: fabien.friedli@unige.ch

\medskip{}

\noindent Anders Karlsson 

\noindent Section de mathématiques \\
Université de Genève

\noindent 2-4 Rue du Lièvre\\
Case Postale 64

\noindent 1211 Genève 4, Suisse 

\noindent e-mail: anders.karlsson@unige.ch 

and

Matematiska institutionen

Uppsala universitet

Box 256

751 05 Uppsala, Sweden

\noindent e-mail: anders.karlsson@math.uu.se
\end{document}